\documentclass[11 pt]{article}

\usepackage[margin=1 in]{geometry}
\usepackage{verbatim}
\usepackage{graphics}
\usepackage[pdftex]{graphicx}
\usepackage{amsmath, amssymb, amsfonts, amsthm}
\usepackage{url}
\usepackage{setspace}
\usepackage{algorithmicx}
\usepackage{enumitem}
\usepackage{color}
\setlist{nolistsep}

\def \PG[#1,#2]{PG(#1,#2)}
\def \AG[#1,#2]{\mathbb{F}_{#2}^{#1}}
\def \N{\mathcal{N}}
\def \cN{\mathcal{N}^c}
\newtheorem{theorem}{Theorem}[section]
\newtheorem{conjecture}[theorem]{Conjecture}
\newtheorem{lemma}[theorem]{Lemma}

\newtheorem{proposition}[theorem]{Proposition}

\newtheorem{fact}{Fact}
\newtheorem{defn}{Definition}

\begin{document}

\title{Finite field Kakeya and Nikodym sets in three dimensions}
\author{Ben Lund\thanks{Rutgers University, supported by NSF grant CCF-1350572.} \and Shubhangi Saraf\footnotemark[1] \and Charles Wolf\footnotemark[1] \\}
\maketitle

\begin{abstract}
We give improved lower bounds on the size of Kakeya and Nikodym sets over  $\mathbb{F}_q^3$.  
We also propose a natural conjecture on the minimum number of points in the union of a not-too-flat set of
lines in $\mathbb{F}_q^3$, and show that this conjecture implies an optimal
bound on the size of a Nikodym set.
\end{abstract}

\section{Introduction}
Let $\mathbb{F}_q$ denote the finite field of $q$ elements.  
A \textit{Kakeya set} $K\subseteq\mathbb{F}_q^n$ is a set of points which contains `a line in every direction.' More precisely, 
for all $x\in\mathbb{F}_q^n$ there is a $y\in\mathbb{F}_q^n$ such that the
line\footnote{A \textit{line}  is an affine subspace of dimension $1$.} $\{xt+y,
t\in\mathbb{F}_q\}\subseteq K$.  

The question of establishing lower bounds for Kakeya sets over finite fields was asked by Wolff \cite{wolff1999recent}. 
In 2008, in a breakthrough result, Dvir~\cite{dvir2009kakeya} showed that for a Kakeya
set $K$ over a finite field $\mathbb{F}$ of size $q$, $|K|>\frac{q^n}{n!}$, thus exactly pinning down the exponent of $q$ in the lower bound.  Later
in 2008, Saraf and Sudan \cite{Saraf2008Improved} improved the lower bound to the
form $1/2\cdot\beta^n q^n$, where $\beta$ is approximately $1/2.6$. Moreover, Dvir showed how to construct a Kakeya set of size $\frac{q^n}{2^{n-1}}+O(q^{n-1})$ (see~\cite{Saraf2008Improved}). 
In 2009, Dvir, Kopparty, Saraf and Sudan \cite{DKSS2009Mergers} proved a lower
bound of $\frac{q^n}{2^n}$ for the size of Kakeya sets.  Thus the gap between the lower bound and the upper bound given by the construction is only at most a factor of 2, and it is a very interesting question to close this gap. Though we now know extremely strong lower bounds, we still do not know an exact bound for any dimension other than 2. 
For $n=2$, Blokhuis and Mazzocca gave exact bounds on the size of a Kakeya set of $q(q + 1)/2 + (q - 1)/2$ for odd $q$ and $q(q + 1)/2$ for even $q$.
In this paper we give improved lower bounds for dimension $n=3$, using an extension of the argument presented in \cite{Saraf2008Improved}. 

A very closely related notion to Kakeya sets is that of Nikodym sets. 
A \textit{Nikodym set} $\N \subseteq \mathbb{F}_q^n$ is
a set of points such that, through each point $p \in \AG[n,q]$, there is a line
$\ell$ such that $\ell \setminus \{p\} \subseteq \N$. 

The best known general lower bounds on the size of Nikodym sets in $\mathbb{F}_q^n$ match the corresponding bounds for Kakeya sets, and follow from the same arguments.
However, we believe that much stronger lower bounds should hold for Nikodym sets; in particular, all known constructions of Nikodym sets have size $(1-o(1))q^n$.
For most choices of $q,n$, it remains an open problem to show that the smallest Nikodym sets are always larger than the smallest Kakeya sets.

In this paper, we give an improved lower bound on the size of Nikodym sets in $\mathbb{F}_q^3$.
Combined with Dvir's construction of Kakeya sets, our bound establishes a separation between the minimum size of Kakeya sets and the minimum size of Nikodym sets in $\mathbb{F}_q^3$.
We also give a new conjecture on the minimum size of a set of lines in $\mathbb{F}_q^3$, not too many of which lie in a plane, and show that this conjecture would imply that the minimum size of a Nikodym set in $\mathbb{F}_q^3$ is $(1-o(1))q^3$.

We now present the relevant background as well as state our results for Kakeya and Nikodym sets. 
All asymptotics are in terms of $q$, which always represents the order of a finite field.  The dimension $n$ is treated as a fixed constant.
For example, $o(1)$ is a function that tends to $0$ as $q$ tends to $\infty$. 

\subsection{Kakeya sets: Background and our results}

We prove the following improved lower bound for Kakeya sets in dimension $n=3.$ 
\begin{theorem}\label{th:Kakeya}
 There exists a constant $C>0$, such that for any prime power $q > C$, if $K \subseteq \mathbb{F}_q^3$ is a Kakeya set, then $$|K|\geq
0.2107q^3.$$ 
\end{theorem}

Prior to this work, the best lower bound for $n=3$ was obtained by Saraf and Sudan \cite{Saraf2008Improved}, and they achieved a lower bound of $(0.208)q^3.$

Though the quantitative improvement in the lower bound is small, we believe our proof method is interesting and might be of independent interest. The proof of Saraf and Sudan \cite{Saraf2008Improved} extended the beautiful polynomials based lower bounds of Dvir \cite{dvir2009kakeya} by using the notion of the multiplicity of roots of polynomials. Our work uses the notion of ``fractional multiplicity" to obtain the improved result. We say a few more words about these proof methods. 

Dvir \cite{dvir2009kakeya} obtained his lower bound via the following argument using polynomials: If the size of $K$ is small, then interpolate a nonzero low degree polynomial $P$ vanishing on all the points of $K$. Then, use the properties of $K$ to show that $P$ must actually vanish at all points of the underlying space\footnote{Actually in this step Dvir uses a polynomial very closely related to $P$, but for simplicity we think of it to be $P$ itself.}. However this contradicts the low degreeness of $P$. 

The work of Saraf and Sudan \cite{Saraf2008Improved} extends this idea by taking a polynomial $P$ that vanishes of each point of $K$ with some higher multiplicity $m$. To enable this, they allow the degree of $P$ to be somewhat higher, but they cap the individual degree of each variable of $P$. This idea leads to stronger bounds than those given by Dvir's argument. 

The novelty of the current work is that we allow the multiplicity $m$ to take a non-integer value. We need to now specify what it means for a polynomial to vanish with multiplicity $m$, where $m$ is a positive real number that is not an integer. For this we define a suitable random process which makes the expected multiplicity of $P$ at a point equal to $m$. By allowing $m$ to take a non-integer value we are able to make finer optimizations. 

We note that in \cite{DKSS2009Mergers} they used higher order derivatives to obtain certain lower bounds on Kakeya sets.  It may be possible to combine higher order derivatives and non-integer multiplicity to obtain bounds better than in \cite{DKSS2009Mergers} and this paper.



A proof of Theorem \ref{th:Kakeya} is in Section \ref{sec:Kakeya}.

\subsection{Nikodym Sets: Background and our results}

The main conjecture in the study of finite {field} Nikodym sets is the following.

\begin{conjecture}\label{conj:o1}
Let $\mathcal{N}$ be a Nikodym set in $\mathbb{F}_q^n$.
Then,
$$|\mathcal{N}| \geq (1-o(1))q^n.$$
\end{conjecture}

Conjecture \ref{conj:o1} is open for general $q$ and $n>2$.
Guo, Kopparty, and Sudan  \cite{guo2013new} proved Conjecture \ref{conj:o1} for all dimensions, but only over fields of constant characteristic.

In the plane, much more is known.
Bounds on the maximal size of minimal blocking sets \cite{bruen1977blocking} can easily be adapted to show that the complement of a Nikodym set in $\PG[2,q]$ has size at most $q^{3/2} + 1$.
In the case that $q$ is not square, Sz\H{o}nyi et. al. \cite{szonyi2005large} improved this bound to $q^{3/2} + 1 - \frac{1}{4}s(1-s)q$, where $s$ is the fractional part of $\sqrt{q}$.
Blokhuis et. al. \cite{blokhuis2015blocking} showed that, if $q$ is a square, then there a Nikodym set in $\PG[2,q]$, the complement of which has size $q^{3/2} - O(q \log q)$.

In Section \ref{sec:3DNikodym}, we prove the following theorem which gives the first separation between the minimum
possible size of Kakeya and Nikodym sets in $\mathbb{F}_q^3$ for arbitrary
sufficiently large prime power $q$.

\begin{theorem}\label{th:3DNikodym}
Let $\mathcal{N}$ be a Nikodym set in $\mathbb{F}_q^3$ with $q$ sufficiently large.
Then,
$$|\mathcal{N}| \geq 0.38 q^3.$$
\end{theorem}
While this falls short of proving the case $n=3$ of Conjecture \ref{conj:o1}, it
does show a separation between Kakeya and Nikodym sets in $\mathbb{F}_q^3$,
since the construction in \cite{Saraf2008Improved} gives a
Kakeya set of size $(0.25 + o(1))q^3$.

\subsubsection{A conjecture on the union of lines}
For $L$ a set of lines, we define $P(L)$ to be the collection of points contained in some line of $L$. More precisely,
$$P(L) = \bigcup_{\ell \in L}\{p \mid p \in \ell\}.$$

In Section \ref{sec:unionOfLines}, we show that a slight modification of the proof of Theorem \ref{th:3DNikodym} shows that if $L$ is any set of $(0.62 + o(1))q^3$ lines in $\mathbb{F}_q^3$, then $|P(L)| \geq (0.38 -o(1)) q^3$. {A similar result appears earlier without explicit constants \cite{Oberlin2014Union}.} Such a result is stronger than Theorem~\ref{th:3DNikodym} since the definition of a Nikodym set guarantees the existence of a set $L$ of lines, one for each point in the complement of the Nikodym set, such that all but one point of each line of $L$ is contained in the Nikodym set.
We  also show that this bound on the union of lines is nearly tight, which stands in contrast to the corresponding bound on the size of a Nikodym set.


The proof of Theorem \ref{th:3DNikodym} uses very little information about $L$ (the set of lines corresponding to the complement of a Nikodym set), and there is more structure that one might be able to exploit in order to get a stronger result. For example, we show in Section \ref{sec:coplanarLines} that no more than $(1+o(1))q^{3/2}$ lines of $L$ can be contained in any plane.
It may be that the approach of bounding the size of the set of lines associated to the complement of a Nikodym set could lead to a proof of Conjecture \ref{conj:o1}, if this additional structure of $L$ is used.

To this end, we propose the following conjecture.

\begin{conjecture}\label{conj:weakUnionOfLines}
	Let $C>0$ be a constant independent of $q$, and let $\alpha(q) \in \omega(q)$.
	If $L$ is a set of at least $Cq^3$ lines and no plane contains $\alpha(q)$ lines of $L$, then $|P(L)| \geq (1-o(1))q^3$.
	The $o(1)$ is a function of $q$ that depends on $C$ and $\alpha$.
\end{conjecture}
In Section \ref{sec:coplanarLines}, we show that Conjecture \ref{conj:weakUnionOfLines} implies the three dimensional case of Conjecture \ref{conj:o1}.

A similar question has been investigated in the context of Kakeya sets.
Wolff \cite{wolff1999recent} showed that if $L$ is a set of $q^2$ lines in $\mathbb{F}_q^3$, no more than $O(q)$ of which lie in any single plane, then $P(L) = \Omega(q^{5/2})$, and this immediately implies the same lower bound on the cardinality of a Kakeya set.
Mockenhaupt and Tao \cite{mockenhaupt2004restriction} observed that, in contrast with the situation for Kakeya sets, Wolff's bound on $P(L)$ is tight, at least when $q$ is square.
{Ellenberg} and Hablicsek \cite{ellenberg2013incidence} improved Wolff's result by characterizing the cases in which it is nearly tight, and in particular showed that, with the same hypotheses on $L$, if $q$ is prime then $P(L) = \Omega(q^3)$.

Although Conjecture \ref{conj:weakUnionOfLines} would be sufficient for an application to Conjecture \ref{conj:o1}, we do not have a counterexample to the following, much stronger, conjecture.

\begin{conjecture}\label{conj:unionOfLines}
Let $\epsilon > 0$ be any constant and let $q$ be a sufficiently large prime power.
Let $L$ be a set of at least $q^{5/2 + \epsilon}$  lines in $\mathbb{F}_q^3$ such that no plane contains more than $(1/2)q^{3/2}$ lines of $L$.
Then, $|P(L)| \geq q^3 - O(q^{5/2})$.
\end{conjecture}

Two examples show that Conjecture \ref{conj:unionOfLines}, if true, is nearly as strong as possible.
Both constructions are based on Hermitian varieties; the definition of a Hermitian variety and the details of the second construction are given in Section \ref{sec:hermitian}.

If $V$ is a Hermitian variety in $\mathbb{F}_q^3$, then $V$ contains $q^2$ lines, no more than $q^{1/2}$ in any plane, and $q^{5/2}$ points; indeed, this is the example used by Mockenhaupt and Tao to show that Wolff's theorem cannot be improved.
Considering the union of $O(q^{1/2})$ such varieties shows that the hypothesis on $|L|$ in Conjecture \ref{conj:unionOfLines} cannot be substantially relaxed.
It may be the case that the union of any set of $q^{1/2 + \varepsilon}$ Hermitian varieties must contain $(1-o(1))q^3$ points in total; for example, analogous bounds are known to hold for affine subspaces \cite{lund2014incidence}.

If $V$ is a Hermitian variety, then through each point $p \in V$ there pass $q- q^{1/2}$ distinct lines that intersect $V$ only at $p$.
By taking a random selection of half of the points in $V$, we can obtain a set $L$ of $\Omega(q^{7/2})$ distinct tangent lines, no more than $(1/2+o(1))q^{3/2}$ of which lie in any plane, and whose union is at most $q^3 - (1/2)q^{5/2}$ points.

A proof of Conjecture \ref{conj:weakUnionOfLines} would be new and very interesting even in the case of prime order fields, for which the above constructions based on Hermitian varieties do not occur.

\section{Kakeya sets in 3 dimensions} \label{sec:Kakeya}
In this section we give a proof of Theorem \ref{th:Kakeya}.

\subsection{Preliminary Results and Lemmas}
Let $\mathbb{F}_q[x_1,...,x_n]=\mathbb{F}_q[\bf{x}]$ be
the ring of polynomials in $x_1,...,x_n$ with coefficients in
$\mathbb{F}_q$.

The following is a basic and well known fact about zeroes of polynomials. 

\begin{fact}\label{thm:SchwarzZippel}
Let $P\in\mathbb{F}_q[\bf{x}]$ be a polynomial of degree at most $q-1$ in each
variable.  If $P(a)=0$ for each $a\in \mathbb{F}_q^n$, then $P\equiv 0.$
\end{fact}

Let $N_q(n,m)$ be the number of monomials in
$\mathbb{F}_q[x_1,...,x_n]$ of individual degree $<q$ and total degree $<mq$. 
Note that $m$ need not be a natural number to define $N_q(n,m)$, rather $m$ can
be any positive real number. 

\begin{lemma}\label{thm:CountingNqnm}
  
$$N_q(n,m)=\sum\limits_{i=0}^{n} (-1)^i{n\choose
i}{\lfloor(m-i)q+n-1\rfloor\choose n},$$ where $\lfloor x\rfloor$ is the largest integer that is at most $x$. 

\end{lemma}

\begin{proof}
The proof will be via inclusion-exclusion. Consider the total number of monomial terms
of a polynomial of total degree strictly less than $mq$.  
This equals $\lfloor mq+n-1\rfloor\choose n$. We only want to include those monomials in our count that have individual degree at most $q-1$.  
Let $C_r$ be the total number of monomials of total degree strictly less than $mq$ and some particular $r$ of the variables having degree $q$ or more. Then by inclusion-exclusion, $$N_q(n,m) = \sum\limits_{i=0}^{n} (-1)^i{n\choose
i} C_i.$$ 

It is not hard to see that $C_i = {\lfloor(m-i)q+n-1\rfloor\choose n}$ since if a particular set of $i$ variables must have degree at least $q$, we can ``peel off" degree $q$ part from each of these variables to get a resulting monomial of total degree at most $\lfloor(m-i)q+n-1\rfloor$. $C_i$ is then then number of such monomials which equals ${\lfloor(m-i)q+n-1\rfloor\choose n}$. 


\end{proof}

\begin{defn}(multiplicity)
For a polynomial $g\in \mathbb{F}_q[\bf{x}]$, we say $g$ vanishes at a point
$\bf{a}$ with multiplicity $m$ if $g(\bf{x+a})$ has no monomial term of degree
lower than $m$.
\end{defn}

The following lemma is a simple adaptation of a lemma from \cite{Saraf2008Improved} (where instead of two sets $S_1$ and $S_2$ there was only one set).

\begin{lemma}\label{thm:Constraints}
Let $m_1 \geq 0$ and $m_2\geq 0$ be integers and $m > 0$ be a real number. Let $S_1, S_2 \subset \mathbb{F}_q^n$ be disjoint sets such that 
$|S_1|{m_1+n-1\choose n}+|S_2|{m_2+n-1\choose n}<N_q(n,m)$. Then there exists a
non-zero polynomial  $g\in \mathbb{F}_q[\bf{x}]$ of total degree less than $mq$
and individual degree at most $q-1$ such that $g$ vanishes on each point of $S_1$
with multiplicity $m_1$ and on $S_2$ with multiplicity $m_2$.  
\end{lemma}

\begin{proof}
The total number of possible monomials in $g$ is $N_q(n,m)$. We consider the coefficients of these monomials to be free variables. 
For each point $\bf{a}\in \mathbb{F}_q^n$, requiring that the polynomial vanishes on $\bf{a}$ with multiplicity $m_i$ adds $m_i+n-1\choose n$ homogeneous linear constraints on these coefficients. 
Requiring that $g$ vanishes on each point of $S_1$
with multiplicity $m_1$ and on $S_2$ with multiplicity $m_2$ imposes a total of $|S_1|{m_1+n-1\choose n}+|S_2|{m_2+n-1\choose n}$ homogeneous linear constraints. Since $|S_1|{m_1+n-1\choose n}+|S_2|{m_2+n-1\choose n}<N_q(n,m)$, thus the total number of  homogeneous linear constraints is strictly less than the number of variables and hence a nonzero solution exists. 
Thus there exists a
non-zero polynomial  $g\in \mathbb{F}_q[\bf{x}]$ of total degree less than $mq$
and individual degree at most $q-1$ such that g vanishes on each point of $S_1$
with multiplicity $m_1$ and on $S_2$ with multiplicity $m_2$.
\end{proof}

For $g\in \mathbb{F}_q[\bf{x}]$ let $g_{\bf{a,b}}(t)=g(\textbf{a}+t\textbf{b})$
denote its restriction to the ``line''
$\{\textbf{a}+t\textbf{b},t\in\mathbb{F}_q\}.$  

The lemma below is a basic result that also appears in~\cite{Saraf2008Improved}. 

\begin{lemma}\label{thm:Univariate}
If $g\in \mathbb{F}_q[\bf{x}]$ vanishes with multiplicity $m$ at some point
$\textbf{a}+t_0\textbf{b}$ then $g_{a,b}$ vanishes with multiplicity m at $t_0$.
\end{lemma}
\begin{proof}
By definition, the fact that $g$ has a zero of multiplicity $m$ at
$\textbf{a}+t_0\textbf{b}$ implies that the polynomial $g(\textbf{x}+\textbf{a}+t_0\textbf{b})$ has no support on
monomials of degree less than $m$.   Thus under the homogeneous substitution of
$\textbf{x}\to t\textbf{b}$, we get no monomials of degree less than $m$ either, and thus we have $t^m$ divides $g(t\textbf{b}+\textbf{a}+t_0\textbf{b})=g(\textbf{a}+(t+t_0)\textbf{b})=g_{\textbf{a,b}}(t+t_0)$.  Hence $g_{\textbf{a,b}}$ has a zero of multiplicity $m$ at $t_0.$

 \end{proof}

The following theorem was the lower bound result from \cite{Saraf2008Improved}.
\begin{theorem}[Kakeya lower bound from~\cite{Saraf2008Improved}]\label{thm:2008saraf}
If K is a Kakeya set in $\mathbb{F}_q^n$, then $|K|\geq\frac{1}{{m+n-1\choose 
n}}N_q(n,m)$. 
\end{theorem}

By setting $n=3$ and $m=2$, it is concluded in \cite{Saraf2008Improved} that for a Kakeya set $K \subseteq \mathbb{F}_q^n$, $|K|\geq \frac{5}{24}q^3\approx 0.2083q^3$. We manage to obtain our strengthened lower bound by allowing $m$ to take values that are not necessarily integers. In particular, we introduce a notion of vanishing with fractional multiplicity and show that it can be used for an improved bound. 

\subsection{Proof of Theorem \ref{th:Kakeya}}
Let $K \subseteq \mathbb{F}_q^3$ be a Kakeya set. As a first step in the proof, we will interpolate a nonzero polynomial vanishing on the points of $K$ with some possibly fractional multiplicity $m$. If we wanted to interpolate a polynomial vanishing with multiplicity $m$ where $m$ is sandwiched between two positive integers $u$ and $u+1$, one way to do this could be that independently for each point we could make it vanish with multiplicity $u$ with some probability, say $\alpha$, and with multiplicity $u+1$ with probability $1-\alpha$, so that in expectation the multiplicity of vanishing would be at least $m$. 
It turns out that the main property of the multiplicities of vanishing we will need is that on each {\it line} of the Kakeya set, almost the correct ($\alpha$) fraction of points have multiplicity of vanishing being at least $u$ and the rest have multiplicity of vanishing at least $u+1$. To do this we will first identify an appropriate subset $S$ of the Kakeya set on which we will want the vanishing multiplicity to be $u$, and in the lemma below we show that such a set can be suitably picked. 

\begin{lemma}\label{th:Chernoff}
Let $K \subseteq \mathbb{F}_q^3$ be a Kakeya set. Let $0 \leq\alpha \leq 1$, and $\delta=\frac{1}{\sqrt[3]{q}}$. Then there exists a constant $C>0$ such that for $q>C$ 
we can pick a subset $S\subset K$ such that $||S|-\alpha|K|| < \delta \alpha |K|$,
and such that for each line $L$ contained in $|K|$, $||L\cap S|-\alpha q| <  \delta
\alpha q$.
\end{lemma}

\begin{proof}
Consider a random subset $S\subset K$, where we choose each point in $S$
independently with probability $\alpha$. By the Chernoff Bound,
$\mathbb{P}[||S|-\alpha |K||\geq\delta\alpha |K|] \leq \exp(-\frac{\alpha
|K|\delta^2}{3})$. Since $|K|$ is certainly larger than $q$,  $\exp(-\frac{\alpha
|K|\delta^2}{3})\leq \exp(-\frac{\alpha q\delta^2}{3})$. 

Note also that there are only $q^4+q^3+q^2$ distinct lines in $\mathbb{F}_q^3$, and thus at most
$q^4+q^3+q^2$ lines in $K$. Let $L$ be any line in $K$. 
Again, via the Chernoff Bound, we have $\mathbb{P}[||L\cap S|-\alpha
q| \geq \delta\alpha q]\leq \exp(-\frac{\alpha q\delta^2}{3})$. By the union bound, the probability that any one of the
lines in $K$ has more than $(1+\alpha\delta)q$ or fewer than $(1-\alpha\delta)q$ points
in $S$ is at most $(q^4+q^3+q^2)\exp(-\frac{\alpha q\delta^2}{3})$.\\ 
Thus if we show that
$\exp(-\frac{\alpha q\delta^2}{3})+(q^4+q^3+q^2)\exp(-\frac{\alpha
q\delta^2}{3})<1$, then by the probabilistic method, such a subset $S$ with the desired properties exists. Since
$\lim\limits_{q\to\infty}\exp({-\frac{\alpha q\delta^2}{3}}
)+(q^4+q^3+q^2)\exp(-\frac{\alpha q\delta^2}{3})=0$ for the appropriately chosen
$\delta$,  there exists some constant $C>0$ such that for $q>C$, there exists such a set $S$.
\end{proof}

\begin{lemma}\label{th:fractional}
Let $K \subseteq \mathbb{F}_q^3$ be a Kakeya set. Let $u\in\{1,2\}$, let $\alpha$ be such that $0\leq\alpha\leq 1$, $\delta= \frac{1}{\sqrt[3]{q}}$ and
$m=(\alpha-\delta\alpha) u+(1-\alpha-\delta\alpha)(u+1).$ Then $$N_q(3,m)\leq
(\alpha+\delta\alpha){2+u\choose 3}|K|+(1-\alpha+\delta\alpha){3+u\choose
3}|K|.$$   
\end{lemma}

\begin{proof}
Suppose for contradiction, $N_q(3,m)> (\alpha+\delta\alpha){2+u\choose
3}|K|+(1-\alpha+\delta\alpha){3+u\choose 3}|K|$. By Lemma \ref{th:Chernoff},
choose $S$ such that each line in $K$ has between $\alpha q-\delta\alpha q$ and
$\alpha q+ \delta\alpha q$ points in $S$ and $||S|-\alpha|K||<\delta \alpha
|K|$.   
In particular $|S|<(\alpha+\delta\alpha)|K|$ and $|K\setminus
S|<(1-\alpha+\delta\alpha)|K|$.
Then by Lemma~\ref{thm:Constraints} there exists a nonzero polynomial
$g\in\mathbb{F}_q[x_1,x_2,x_3]$ with total degree less than $mq$ and individual
degree less than $q$ such that $g$ vanishes on $S$ with multiplicity at least $u$ and on
$K\setminus S$ with multiplicity at least $u+1$. Let $d$ denote the degree of $g$.  Let
$g=g_0+g_1$, where $g_0$ denotes the homogeneous part of degree $d$ and $g_1$
the part with degree less than $d$.  Note that $g_0$ also has degree at most
$q-1$ in each of its variables.
\\Now fix a ``direction'' $\textbf{b}\in \mathbb{F}_q^3$.  Since $K$ is a Kakeya
set, there exists $\textbf{a}\in \mathbb{F}_q^3$ such that the line
$\textbf{a}+t\textbf{b}\in K$ for all $t\in\mathbb{F}_q$.  So consider
$g_{a,b}(t)$, the univariate polynomial of $g$ restricted to the line
$\textbf{a}+t\textbf{b}$.
By Lemma~\ref{th:Chernoff} and Lemma~\ref{thm:Univariate}, there are at least $(1-\delta)\alpha q$ choices of $t$ where $g_{a,b}$
vanishes with multiplicity at least $u$ and there are at least $q -\alpha q-\delta\alpha
q$ choices of $t$, where $g_{a,b}$ vanishes with multiplicity at least $u+1$.  So in total, $g_{a,b}$ has at least $(\alpha-\delta\alpha)uq+(1-\alpha-\delta\alpha)(u+1)q=mq$ zeros, which is more
than its degree.  Therefore, $g_{a,b}$ must be identically zero.  In particular,
its leading coefficient must be 0.  Since this leading coefficient equals
$g_0(\textbf{b})$, $g_0(\textbf{b})=0$. Since $b$ was chosen arbitrarily, this must happen for all $\textbf{b}\in \mathbb{F}_q^3$. 
However, by Fact~\ref{thm:SchwarzZippel}, this contradicts the fact that $g_0$ is a nonzero polynomial of degree at most $q-1$ in each of its variables.
\end{proof}

\begin{proof}[Proof of Theorem \ref{th:Kakeya}]
Let $\delta= \frac{1}{\sqrt[3]{q}}$, let $u\in\{1,2\}$, let $\alpha$ be such that $0\leq\alpha\leq 1$, and
$m=(\alpha-\delta\alpha) u+(1-\alpha-\delta\alpha)(u+1).$  Note that once we set the value for $u$ and $m$ between $1$ and $2$, this will determine a value for $\alpha$. For now suppose we have chosen some values for $u$, $\alpha$ and $m$.  

By Lemma
\ref{th:fractional}, $|K|\geq \frac{N_q(3,m)}{(\alpha+\delta\alpha){2+u\choose
3}+(1-\alpha+\delta\alpha){3+u\choose 3}}$.  Since we are considering $|K|$ as
$q$ grows asymptotically, we only need to consider the leading term when
$N_q(3,m)$ is expressed as a polynomial in $q$.  Also, note that $\delta$ becomes small as
$q$ grows large.    

The reason we only let $u$ take value $1$ or $2$ is the following. Since we only care about polynomials with individual variable degree less than $q$, the total degree must be less than $3q$. Choosing a value of $m$ that is greater than or equal to $3$ will just end up being somewhat redundant and end up giving a worse bound. Thus we only consider $m < 3$. Given the relationship between $u$ and $m$ and given that $u$ needs to be an integer, the only choices for $u$ are hence $1$ or $2$ as in the statement of the above lemma. 

When $u=1$, this makes $m=2-(1+o(1))\alpha$ for large $q$.  By Lemma \ref{thm:CountingNqnm},
$$N_q(3,m)=\left(\frac{-2m^3+9m^2-9m+3}{6}+o(1)\right)q^3.$$ Substituting $u=1$, by Lemma \ref{th:fractional} we get that
$$|K|\geq\left(\frac{-2m^3+9m^2-9m+3}{6(4-3\alpha)}+o(1)\right)q^3=\left(\frac{
-2m^3+9m^2-9m+3}{6(3m-2)}+o(1)\right)q^3.$$  We maximize this for $1\leq
m\leq 2$.  For m=1.84, this gives $|K|\geq (0.21076+o(1))q^3.$
When $u=2$, the best lower bound achieved in this case is $|K|\geq (.2083+o(1))q^3$. Thus overall the best 
lower bound we achieve is $(0.21076+o(1))q^3$.
\end{proof}


\section{Nikodym sets in 3 dimensions and the union of
lines}\label{sec:3DNikodym}

In this section, we investigate  Nikodym sets in $\mathbb{F}_q^3$ and give improved lower bounds.

We will find it easier to work with the complement of a Nikodym set rather than the Nikodym set itself.
We define
$$f(n,q) = \text{ the maximum size of the complement of a Nikodym set in }\mathbb{F}_q^n.$$
We additionally denote the complement of a set $\mathcal{N}$ by $\mathcal{N}^c$.

Using this notation, Conjecture \ref{conj:o1} states that $f(n,q) = o(q^n)$, and Theorem \ref{th:3DNikodym}  states that $f(3,q) \leq (0.62 + o(1))q^3$.

In Section \ref{sec:3Dproof}, we prove Theorem \ref{th:3DNikodym}.

In Section \ref{sec:unionOfLines}, we show that the proof of Theorem
\ref{th:3DNikodym} given in Section \ref{sec:3Dproof} immediately implies a
lower bound on the number of points incident to a large set of lines, and that
this bound is nearly tight.
This implies that any substantial improvement to Theorem \ref{th:3DNikodym} will
need to use some property of Nikodym sets that is not exploited by the proof
given in Section \ref{sec:3Dproof}.

In Section \ref{sec:coplanarLines}, we observe the set of lines associated to the complement of a Nikodym set has the
property that not too many of the lines given by its definition can lie in any
single plane.
We further suggest that exploiting this property might lead to a proof of
Conjecture \ref{conj:o1} in the three dimensional case.
In particular, we show that a proof of Conjecture \ref{conj:weakUnionOfLines} would
immediately imply the case $n=3$ of Conjecture \ref{conj:o1}.

In Section \ref{sec:hermitian} we describe a near-counterexample to Conjecture \ref{conj:unionOfLines}.

\subsection{Proof of Theorem \ref{th:3DNikodym}} \label{sec:3Dproof}

Our bound on $f(3,q)$ will use a bound on the number of incidences
between points and lines.
The bound we will use was essentially proved by Lund and Saraf in \cite{lund2014incidence},
but is not explicitly stated there; a similar bound was obtained by Bennett, Iosevich, and Pakianathan \cite{bennett2014three}.
We show how to recover the bound from arguments given in
\cite{lund2014incidence}.

Given a set $P$ of points and a set $L$ of lines, we denote the number of
incidences between $P$ and $L$ as
\[I(P,L) = |\{(p,\ell) \in P \times L \mid p \in \ell \} |.\]

\begin{theorem}\label{th:incidenceBound}
Let $L$ be a set of lines and $P$ a set of points in $\mathbb{F}_q^3$.
Then,
\[I(P,L) \leq (1 + o(1)) \left ( |P||L|q^{-2} + q
\sqrt{|P||L|(1-|P|q^{-3})(1-|L|q^{-4})} \right ).\]
\end{theorem}
\begin{proof}

A $(d_U, d_V)$-biregular graph $G$ is a bipartite graph such that each each left
vertex has degree $d_U$ and each right vertex has degree $d_V$.
We denote by $e(G)$ the number of edges in a graph $G$, and by $G(S,T)$ the
number of edges between two subsets $S,T$ of the vertices of a graph.
We will use the expander mixing lemma \cite{alon1988explicit}, specifically the following 
bipartite version.
A proof of Lemma \ref{th:expanderMixing} is given in \cite{lund2014incidence}, and an equivalent result was proved much earlier by Haemers, e.g. \cite{haemers1995interlacing}.
\begin{lemma}[Bipartite expander mixing lemma, \cite{lund2014incidence}]\label{th:expanderMixing}
Let $G$ be a $(d_U, d_V)$-biregular graph with left vertices $U$ and right
vertices $V$.
Let $A$ be the (square) adjacency matrix of $G$, and let $\lambda_1 \geq
\lambda_2 \geq \ldots \geq \lambda_{|U|+|V|}$ be the eigenvalues of $A$.
Let $\lambda = \lambda_2 / \lambda_1$.
Let $S \subseteq U$ with $|S| = \alpha |U|$ and let $T \subseteq V$ with $|T| =
\beta|V|$.
Then,
\[\left | \frac{e(S,T)}{e(G)} - \alpha \beta \right | \leq \lambda \sqrt{\alpha
\beta (1- \alpha)(1- \beta)}.\]
\end{lemma}

Construct a bipartite graph $G$ with left vertices $U$ being the points of
$\mathbb{F}_q^3 $, and right vertices $V$ being the lines of $\mathbb{F}_q^3$, with
$(p,\ell)$ in the edge set of $G$ if and only if $p \in \ell$.
The number of points in $\mathbb{F}_q^3$ is $|U| = q^3$; the number of lines is
$|V| = (1+o(1))q^4$; and the number of incidences between points and lines in
$\mathbb{F}_q^3$ is $e(G) = (1+o(1))q^5$.
It is shown in Section 4 of \cite{lund2014incidence} that the largest eigenvalue
of this graph is $(1+o(1))q^{3/2}$, and the second largest eigenvalue is $(1 +
o(1))q$.
We are interested in the number of incidences between a set $P \subseteq U$ and $L \subseteq V$.
This is exactly the number of edges between $P$ and $L$ in $G$, and hence we apply Lemma \ref{th:expanderMixing} with $\alpha = |P|q^{-3}$ (which is the density of $P$ in $U$) and $\beta = (1-o(1)) |L| q^{-4}$ (which is the density of $L$ in $V$), to get that

$$
\left | (1+o(1))(I(P,L)q^{-5} -|L||P| q^{-7}) \right | \leq (1+o(1))q^{-4}
\sqrt{|P||L| (1- |P|q^{-3})(1- |L|q^{-4})}.$$
Thus, simplifying we get 
$$I(P,L) \leq (1+o(1))  \left ( |P||L|q^{-2} + q
\sqrt{|P||L|(1-|P|q^{-3})(1-|L|q^{-4})} \right ).$$

\end{proof}

Now, we complete the proof of Theorem \ref{th:3DNikodym}.

\begin{proof}[Proof of Theorem \ref{th:3DNikodym}]
Suppose that $\cN$ is the complement of a Nikodym set in $\mathbb{F}_q^3$.
Let $L$ be a set of $|\cN|$ lines such that each line has exactly one point in
common with $\cN$, and there is exactly one line of $L$ through each point of
$\cN$; the existence of such a set is guaranteed by the definition of a 
Nikodym set.
Let $P = \N$; by definition, $|P| = q^3 - |L|$.
Then each line of $L$ is incident to exactly $q-1$ points of $P$, so $I(P,L) =
(q-1)|L|$.
Applying Theorem \ref{th:incidenceBound}, we get that

$$
(q-1)|L| \leq (1+o(1)) \left ( (q^3-|L|) |L| q^{-2} + q \sqrt{(q^3 - |L|)
|L|(|L|q^{-3})} \right ).$$

Substituting $|L| = cq^3$ in the above expression, we have
$$
cq^4 - cq^3 \leq (1+o(1)) \left ( (1-c)cq^{4} + (1-c)^{1/2}c q^4 \right).
$$
Rearranging and simplifying,
$$
c^3 + c^4 \leq (1+o(1)) c^2.
$$

Solving for $c$, we find that 
$$c \leq (\sqrt{5}-1)/2+ o(1) \leq (1+o(1)) 0.62.$$

The result of this calculation can easily be checked by setting $|L|/q^3$ to be any constant greater than $0.62$ in the first inequality, which yields a contradiction for $q$ sufficiently large.
\end{proof}

\subsection{The union of lines}\label{sec:unionOfLines}

The proof of Theorem \ref{th:3DNikodym} only uses the fact that the definition
of a Nikodym set $N$ guarantees the existence of $|\N^c|$ distinct lines, each of
which are incident to at least $q-1$ points of $\N$.
While we do not believe that Theorem \ref{th:3DNikodym} is anywhere near tight, the same proof
gives a nearly tight lower bound on the size of the union of any set of at least $0.62 q^3$
lines.

Recall from the introduction that, for any set $L$ of lines,
$$ P(L) = \bigcup_{\ell \in L}\{p \mid p \in \ell\}.$$

\begin{proposition}\label{th:generalUnionOfLines}
If $L$ is a set of $0.62 q^3$ lines in $\mathbb{F}_q^3$, then $|P(L)| \geq (1-o(1))0.38 q^3$.
\end{proposition}

\begin{proof}
Since each point on any line in $L$ is contained in $P = P(L)$, the number of incidences between $L$ and $P$ is $q|L| = 0.62 q^4$.
Applying Theorem \ref{th:incidenceBound},
\begin{align*}
0.62 q^4 & \leq (1+o(1)) (0.62 |P| q + q \sqrt{0.62 |P| q^3 (1 - |P| q^{-3})}), \text{ so, simplifying as before,}\\ 
|P| &> (1 - o(1)) 0.38 q^3.
\end{align*}
\end{proof}

We now show that without any further condition on the set of lines, Proposition \ref{th:generalUnionOfLines} is nearly tight.

\begin{proposition}\label{th:generalUnionOfLinesIsTight}
There is a set $L$ of $(1-o(1))0.62 q^3$ lines in $\mathbb{F}_q^3$ such that $|P(L)| < 0.43 q^3$.
\end{proposition}

\begin{proof}
Let $p$ be an arbitrary point of $\mathbb{F}_q^3$.
We show below that we can choose a set $\Pi$ of $0.62 q$ planes incident to $p$, such
that no line is contained in $3$ planes of $\Pi$.
The set $L$ will be the set of all lines contained in the union of the planes of $\Pi$.
By inclusion-exclusion, the total number of lines chosen is $|L| \geq 0.62 q^3 - \binom{0.62 q}{2} =
(1-o(1)) 0.62 q^3$, and the total number of points on these lines is $(0.62q^3-1) - (q-1) \binom{0.62 q}{2} + 1 < 0.43 q^3$, for $q$ sufficiently large.

To choose the set $\Pi$, we first project from the point $p$; this is a map from the lines incident to $p$ to points in $\PG[2,q]$, the projective plane over $\mathbb{F}_q$.
In this projection, each plane incident to $p$ corresponds to a line in $\PG[2,q]$.
A conic in $\PG[2,q]$  is a set of $q+1$ points, no three collinear; the
projective dual to a conic is a set of $q+1$ lines, no three coincident.
By choosing $\Pi$ to be an arbitrary subset of size $0.62 q$ among the planes associated to such a set of
lines, we ensure that no three contain a common line.
\end{proof}

\subsection{Coplanar lines and Conjecture
\ref{conj:weakUnionOfLines}}\label{sec:coplanarLines}

A consequence of the near tightness of Proposition \ref{th:generalUnionOfLines} is
that any substantial improvement to Theorem \ref{th:3DNikodym} must use some
additional information about Nikodym sets, beyond the fact that the definition
of a Nikodym set $\N$ guarantees the existence of $|\N^c|$ distinct lines, each
incident to $q-1$ points of $\N$.
One such property is that no plane can contain too many of the lines associated
to the complement of a Nikodym set.
{ This property is captured in Proposition \ref{thm:notTooManyLinesInAPlane}.}

{
	Before giving Proposition \ref{thm:notTooManyLinesInAPlane}, we establish a simple lemma related to planar Nikodym sets.
	The following Lemma is not new, see for example \cite{bruen1977blocking, feng2010some}.
	For the reader's convenience, we include a proof of the precise formulation we need.
	We call a line that contains exactly one point of any set a {\em tangent line}.
\begin{lemma}\label{le:feng}
	Let $S \subset \mathbb{F}_q^2$ so that there is at least one tangent line at each point of $S$.
	Then, $|S| \leq q^{3/2}$.
\end{lemma}
\begin{proof}
	Label all of the lines of $\mathbb{F}_q^2$ except for one tangent at each point of $S$ as $\ell_1, \ldots, \ell_{q^2 + q - |S|}$, and let $k_i$ be the number of points of $S$ on line $\ell_i$.
	Since each point of $S$ is contained in $q+1$ lines, of which all but one are enumerated, we have
	\[\sum k_i = q|S|. \]
	Note that each pair of points of $S$ is in exactly one of the enumerated lines, hence
	\[\sum k_i(k_i - 1) = |S|(|S| - 1). \]
	Combining these with an application of Cauchy-Schwarz,
	\[|S|^2 + (q-1)|S| = \sum k_i^2 \geq q^2|S|^2(q^2 + q - |S|)^{-1}. \]
	After some straightforward rearrangement, this implies
	\[q^3 \geq |S|^2, \]
	and hence the result.
\end{proof}

}
\begin{proposition}\label{thm:notTooManyLinesInAPlane}
Let $\N \subseteq \mathbb{F}_q^3$ be a Nikodym set.
Let $L$ be a set of lines, such that each line of $L$ is incident to exactly one
point of $\N^c$, and each point of $\N^c$ is incident to exactly one line of
$L$.
Then any plane in contains at most $q^{3/2}$ lines of $L$.
\end{proposition}
Note that the existence of a set satisfying the conditions on $L$ in this proposition is guaranteed by the definition of a Nikodym set.

\begin{proof}
Let $\pi$ be a plane, and let $L'$ be the subset of lines of $L$ that are contained
in $\pi$.
Let $P \subseteq \N^c$ be the set of points associated to lines in $L'$.
{ $P$ satisfies the conditions of Lemma \ref{le:feng}.}
Hence, $|L'| = |P| \leq q^{3/2}$.
\end{proof}

The observation recorded in Proposition \ref{thm:notTooManyLinesInAPlane} enables us to show that Conjecture \ref{conj:weakUnionOfLines} implies the three dimensional case of Conjecture \ref{conj:o1}.
Since Proposition \ref{thm:notTooManyLinesInAPlane} only gives an upper bound of $(1+o(1))q^{3/2}$ lines contained in any plane, while Conjecture \ref{conj:weakUnionOfLines} requires a bound of any function in $\omega(q)$, we will need to use some additional incidence theory to bridge the gap.
In particular, we will use the following lemma, which is a special case of Corollary 6 in \cite{lund2014incidence}.

\begin{lemma}[\cite{lund2014incidence}]\label{thm:planesLines}
For $k > 1$, a set of $kq$ planes in $\mathbb{F}_q^3$ is incident to at least $(1-\frac{1}{k-1+k^{-1}})q^3$
points.
A set of $kq$ lines in $\mathbb{F}_q^2$ is incident to {at least} $(1-\frac{1}{k-1+k^{-1}})q^2$
points.
\end{lemma}

We now prove that Conjecture \ref{conj:weakUnionOfLines} implies the three
dimensional case of Conjecture \ref{conj:o1}.

\begin{theorem}\label{th:unionOfLinesImplieso1}
If Conjecture \ref{conj:weakUnionOfLines} holds, then the case $n=3$ of Conjecture
\ref{conj:o1} holds.
\end{theorem}

\begin{proof}
Suppose that Conjecture \ref{conj:weakUnionOfLines} holds.

Let $\cN$ be the complement of a Nikodym set in $\mathbb{F}_q^3$.
Let $L$ be a set of lines such that each line of $L$ is incident to exactly one
point of $\cN$, and each point of $\cN$ is incident to exactly one line of $L$;
the existence of such a set is guaranteed by the definition of a Nikodym set.
Let $L_1 \subset L$ be an arbitrary subset of $\lfloor |L|/2 \rfloor$ lines of
$L$, and let $P \subset \cN$ be the set of points in $\cN$ that are not incident to any line in $L_1$.

Let $\alpha(q) \in \omega(q)$, and
let $\Pi$ be the set of planes that contain more than $\alpha(q)$ lines
of $L_1$.
Let $L_2 \subseteq L_1$ be the subset of lines in $L_1$ that are each contained in
some plane of $\Pi$.

Suppose that $|L_2| = \Omega(q^{5/2}\log(q))$.
Since each plane $\pi \in \Pi$ contains at least $\alpha(q)$ lines of
$L_2$, Lemma \ref{thm:planesLines} implies that the probability that a uniformly chosen point of $\pi$ is not on any line of $L_2$ is bounded above by $(1+o(1))q/\alpha(q)$.
By Proposition \ref{thm:notTooManyLinesInAPlane}, no plane of $\Pi$ contains
more than $(1+o(1))q^{3/2}$ lines of $L_2$; hence, $|\Pi| \geq (1-o(1))q^{-3/2}|L_2| = \Omega(q \log q)$.
By Lemma \ref{thm:planesLines}, the probability that a uniformly chosen point of $\mathbb{F}_q^3$ is not on any plane of $\Pi$ is bounded above by $O(1/\log(q))$.
By a union bound, all but $O(q^3/\log(q) + q^4/\alpha(q)) = o(q^3)$ points of $\mathbb{F}_q^3$ are
contained in some line of $L_2$.
By construction, half of the points of $\cN$ are not in any line of $L_1$, and
hence $|\cN| = o(q^3)$.

Now, suppose that $|L_2| = O(q^{5/2} \log q) = o(q^3)$.
By construction, no plane contains more than $\alpha(q)$ lines of $L_1
\setminus L_2$.
Hence, Conjecture \ref{conj:weakUnionOfLines} implies that either $|L_1 \setminus L_2| = o(q^3)$, and hence $|\cN | = o(q^3)$, or $|P(L_1 \setminus L_2)| = (1-o(1))q^3$, and hence $|\cN | = o(q^3)$.
\end{proof}

\subsection{Hermitian varieties}\label{sec:hermitian}

In this section, we describe a construction, mentioned in the introduction, showing that Conjecture \ref{conj:unionOfLines}, if true, is nearly as strong as possible.
We rely on the classical results of Bose and Chakravarti \cite{bose1966hermitian}.

Let $q = p^2$, for $p$ a prime power.
For $v \in \mathbb{F}_q$, we define the conjugate $\overline{v} = v^p$.
Since $q$ has order $p^2$, we have $\overline{\overline{v}} = v$.
We will use homogenous coordinates to represent a point $v \in \PG[n,q]$ as a column vector $\mathbf{v} = (v_0, v_1, \ldots, v_n)^T$.

A square matrix $H = ((h_{ij}))$ for $i,j = 0,1,\ldots,n$ and $h_{ij} \in \mathbb{F}_q$ is \textit{Hermitian} if $h_{ij} = \overline{h_{ji}}$ for all $i,j$.
Let $\mathbf{x}^T = (x_0, x_1, \ldots, x_n)$ and $\overline{\mathbf{x}} = (\overline{x_0}, \overline{x_1}, \ldots, \overline{x_n})^T$.
The set of points $x$ in $\PG[n,q]$ whose coordinates satisfy $\mathbf{x}^T H \overline{\mathbf{x}} = 0$ for a Hermitian matrix $H$ is a \textit{Hermitian variety}.
The \textit{rank} of the Hermitian variety $V$ defined by $\mathbf{x}^T H \overline{\mathbf{x}} = 0$ is defined to be the rank of $H$.
We say that $V$ is \textit{non-degenerate} if its rank is $n+1$.

Let $V$ be a rank $r$ Hermitian variety in $\PG[n,q]$ defined by $\mathbf{x}^T H \overline{\mathbf{x}} = 0$.
A point $c$ of $V$ is \textit{singular} if $\mathbf{c}^T H = \mathbf{0}$.
Clearly, if $V$ is non-degenerate, it has no singular points.
Otherwise, $\mathbf{c}^T H = 0$ has $n - r + 1$ independent solutions, and hence defines an $(n-r)$-flat, which we term the \textit{singular space} of $V$.

The set of points corresponding to row vectors $\mathbf{x}^T$ that satisfy the equation $\mathbf{x}^T H \overline{\mathbf{c}}=0$ is the \textit{tangent space} at $\mathbf{c}$.
If $\mathbf{c}$ is singular, this is the entire space; otherwise, $H \overline{\mathbf{c}}$ is a non-zero vector, and hence the tangent space is a hyperplane.

\begin{lemma}[Section 7 in \cite{bose1966hermitian}]\label{thm:flatSpaceIntersectionWithHermVar}
The intersection of a Hermitian variety with a flat space is a Hermitian variety.
In particular, a line intersects a Hermitian variety in a single point, $q^{1/2}+1$ points, or is entirely contained in the variety.
\end{lemma}

Given a Hermitian variety $V$, we define  \textit{tangent lines} to be those lines that intersect $V$ in exactly $1$ point.

\begin{theorem}[Theorem 7.2 in \cite{bose1966hermitian}]\label{thm:lineContainedinHermVar}
If $V$ is a degenerate Hermitian variety of rank $r < n + 1$, and $c$ is a point belonging to the singular space of $V$, and $d$ is an arbitrary point of $V$, then each point on the line $cd$ belongs to $V$.
\end{theorem}

\begin{theorem}[Theorem 7.4 in \cite{bose1966hermitian}]\label{thm:tangentOfHermVar}
If $V$ is a non-degenerate Hermitian variety, the tangent hyperplane at a point $c$ of $V$ intersects $V$ in a degenerate Hermitian variety $U$ of rank $n-1$. The singular space of $U$ consists of the single point $c$.
\end{theorem}

\begin{theorem}[Theorem 8.1 in \cite{bose1966hermitian}]\label{thm:numPointsHermVar}
The number of points on a non-degenerate Hermitian variety is
$$\phi(n,q) = (q^{(n+1)/2} - (-1)^{n+1})(q^{n/2} -(-1)^n) (q - 1)^{-1}.$$
The number of points on a degenerate Hermitian variety of rank $r$ is
$$(q^{n-r+1}-1)\phi(r-1,q) + (q^{n-r+1} - 1)(q - 1)^{-1} + \phi(r-1,q).$$
\end{theorem}

The parameters of the construction mentioned in the introduction follow by taking $\alpha = 1/2$ in the following proposition.

\begin{proposition}\label{thm:strongUnionOfLinesHermVarConstruction}
Let $q = p^2$ for a prime power $p$, and let $0 < \alpha < 1$.
Then, there is a set $L$ of $(\alpha + o(1)) q^{7/2}$ lines in $\mathbb{F}_q^3$ such that no plane contains more than $(\alpha + o(1))q^{3/2}$ lines of $L$, and {$|P(L)| \leq q^3 - (1 - \alpha + o(1))q^{5/2}$}.
\end{proposition}
\begin{proof}
Let $V$ be a non-degenerate Hermitian variety in $\PG[3,q]$.
By Theorem \ref{thm:numPointsHermVar}, we have $|V| = (1+o(1))q^{5/2}$.
Let $P$ be a set of $\lfloor \alpha |V| \rfloor$ of the points of $V$, chosen uniformly at random.
Let $L$ be the set of tangent lines to $V$ at points of $P$.
Since the tangent lines intersect $V$ only at their points of tangency, it is clear that the $\lceil (1-\alpha)|V| \rceil = (1-\alpha+o(1))q^{5/2}$ points of $V \setminus P$ are not incident to any line of $L$.
It remains to show $|L| = (\alpha + o(1)) q^{7/2}$, and that no plane contains more than $(\alpha+o(1))(q^{3/2})$ lines of $L$.

By Theorem \ref{thm:tangentOfHermVar}, the tangent plane $\Sigma$ to $V$ at an arbitrary point $c \in P$ intersects $V$ in a rank $2$ Hermitian variety $U \subseteq \Sigma$, having the single singular point $c$.
From the second part of Theorem \ref{thm:numPointsHermVar}, we have that $U$ contains $q^{3/2} + q + 1$ points.
Together with Theorem \ref{thm:lineContainedinHermVar}, this implies that $U$ is the union of $q^{1/2} + 1$ lines coincident at $c$.
The remaining $q-q^{1/2}$ lines contained in $\Sigma$ and incident to $c$ are tangent lines to $V$.
Hence, $L$ consists of $(q-q^{1/2})|P| = (\alpha + o(1))q^{7/2}$ distinct lines, and tangent planes to $V$ each contain at most $q-q^{1/2}$ lines of $L$.

By Lemma \ref{thm:flatSpaceIntersectionWithHermVar}, the intersection of a plane $\Sigma$ with $V$ is a Hermitian variety $U$; if $\Sigma$ is not tangent to $V$, then $U$ is non-degenerate.
By Theorem \ref{thm:numPointsHermVar}, we have that $|U|=q^{3/2}+q+1$, and there is a single tangent line at each of these points.
In addition, a line of $L$ will be contained in $\Sigma$ only if it is tangent to one of the points of $U$.
Hence, in order to show that no plane contains more than $(\alpha+o(1))q^{3/2}$ lines of $L$, it suffices to show that no plane contains more than $(\alpha + o(1)) q^{3/2}$ points of $P$.

The expected number of points of $P$ on $\Sigma$ is $\alpha |U|$.
Since the points of $P$ are chosen uniformly at random, the Chernoff bound for Bernoulli random variables implies that, for any $0 < \delta < 1$, the probability that we have more than $(1+\delta) \alpha |U|$ points of $P$ on $\Sigma$ is bounded above by $e^{-\delta^2 \alpha |U| / 3}$.
Taking a union bound over the $(1+o(1))q^3$ planes in $\PG[3,q]$, we have that the probability that any plane has more than $(1 + \delta) \alpha |U|$  points of $P$ is bounded above by $(1+o(1))q^{3}e^{-(1+o(1))\delta^2 \alpha q^{3/2}/3}$.
Hence, taking $\delta > (1+o(1)) 9 \alpha^{-1} q^{-3/4} \log q = o(1)$ ensures that this happens with probability strictly less than $1$, and hence there is a choice of $P$ such that there are fewer than $(\alpha+o(1)) q^{3/2}$ on any plane.

\end{proof}

\section{Acknowledgment}
We thank Anurag Bishnoi for comments on an earlier version of this paper.

\bibliographystyle{plain}
\bibliography{KakeyaandNIkodym}

\end{document}